\documentclass[a4paper,10pt]{article}
\usepackage[english]{babel}
\usepackage[latin1]{inputenc}
\usepackage[T1]{fontenc}
\usepackage{amsmath}
\usepackage{amsfonts}
\usepackage{amsthm}
\usepackage{amssymb}
\usepackage{nccmath}
\usepackage{pifont}
\usepackage{bbm}
\usepackage{tikz}
\usepackage{enumerate}
\usepackage{array}
\usepackage[small]{caption}
\usepackage[round]{natbib}

\newtheorem{theorem}{Theorem}[section]
\newtheorem{prop}{Proposition}[theorem]
\newtheorem{lemma}[theorem]{Lemma}
\newtheorem{dfn}{Definition}
\newcommand{\esp}[2]{\mathbb{E}_{#1} \left( #2 \right)}
\newcommand{\espc}[2]{\mathbb{E}_{#1} \left[ #2 \right]}

\title{Classification and estimation in the Stochastic Block Model based on the empirical degrees}
\author{\textsc{Antoine Channarond}\\
\textsc{Jean-Jacques Daudin} \\
\textsc{and St\'ephane Robin} \\
\\
\textit{Département de Math\'ematiques et Informatique Appliqu\'ees} \\
\textit{UMR 518 AgroParisTech/INRA} \\
\textit{16 rue Claude Bernard, 75231 Paris Cedex 5, France} \\
\\
{\fontfamily{cmbr}\fontseries{m}\selectfont channarond,daudin,robin@agroparistech.fr } }
\date{}

\begin{document}

\maketitle

\begin{abstract}
The Stochastic Block Model \citep{holland1983stochastic} is a mixture model for heterogeneous network data. Unlike the usual statistical framework, new nodes give additional information about the previous ones in this model. Thereby the distribution of the degrees concentrates in points conditionally on the node class. We show under a mild assumption that classification, estimation and model selection can actually be achieved with no more than the empirical degree data. We provide an algorithm able to process very large networks and consistent estimators based on it. In particular, we prove a bound of the probability of misclassification of at least one node, including when the number of classes grows.
\end{abstract}

\section{Introduction}
Strong attention has recently been paid to network models in many domains such as social sciences, biology or computer science. Networks are used to represent pairwise interactions between entities. For example, sociologists are interested in observing friendships, calls and collaboration between people, companies or countries. Genomicists wonder which gene regulates which other. But the most famous examples are undoubtedly the Internet, where data traffic involves millions of routers or computers, and the World Wide Web, containing millions of pages connected by hyperlinks. A lot of other examples of real-world networks are empirically treated in \citet{albert2002statistical}, and book \citet{wasserman1994social} gives a general introduction to mathematical modelling of networks, and especially to graph theory.


One of the main features expected from graph models is inhomogeneity. Some articles, e.g. \citet{bollobas2007phase} or \citet{van2009random}, address this question. In the Erd\H{o}s-Rényi model introduced by \citet{erdos1959random} and \citet{gilbert1959random}, all nodes play the same role, while most real-world networks are definitely not homogeneous.

In this paper, we are interested in the Stochastic Blockmodel (SBM), introduced by \citet{holland1983stochastic} and inspired by \citet{holland1981exponential} and \citet{fienberg1981categorical}. This model assumes discrete inhomogeneity in the underlying social structure of the observed population: $n$ nodes are split into $Q$ homogeneous classes, called blocks, or more generally clusters. Then it is assumed that the distribution of the edge between two nodes, depends only on the blocks to which they belong. Thereby, within each class, all nodes have the same connection behavior: they are said to be structurally equivalent \citep{lorrain1971structural}. When the class assignment is known, the social structure can possibly be visualized through the meta-graph \citep{picard2009deciphering}, which emphasizes the role of each class. However the block structure is supposed to be not observed or \emph{latent}. Thus the assignment $Z$ and the model parameters must be estimated \emph{a posteriori} through the observed graph $X$, which is a real challenge, especially in large networks. 

 Our main purpose in this paper is to present a consistent inference method under SBM, which can above all process very large graphs. \citet{snijders1997estimation} have proposed a maximum likelihood estimate based on the EM algorithm for very small graphs with $Q=2$ blocks. They have also proposed a Bayesian approach based on Gibbs sampling for larger graphs (hundreds of nodes), which they have extended to arbitrary block numbers in \citet{nowicki2001estimation}. However the usual techniques enables the processing of only relatively small graphs, because they suffer severely from graph intricacy. In particular the EM algorithm deals with the conditional distribution of the labels $Z$ given the observations $X$, whose dependency graph is actually a clique in the case of SBM (see paragraph 5.1 in \citet{daudin2008mixture}). Inspired by \citet{wainwright2008graphical}, \citet{daudin2008mixture} have developed approximate methods using variational techniques in the context of SBM. From a physical point of view, the variational paradigm amounts to mean-field approximation, see \citet{jaakkola2000tutorial}. Thus thousands of nodes can be processed with this variational EM algorithm. Lastly, \citet{celisse2011consistency} proves the variational method to be consistent precisely under SBM.

All previous methods treat both classification and parameter estimation directly and at the same time. They are alternatively updated at each step of EM-based algorithms. Yet those tasks are actually not symmetrical, and moreover estimators are quite simple when $Z$ is known. The classification ---~remaining the main pitfall thus far~--- can be completed first, and then the latent assignment $Z$ just replaced with this classification by plug-in in order to estimate the parameters.

Searching for clusters from a graph is computationally difficult and has different meanings. Many algorithms, especially coming from physics and computer science, aim at detecting highly connected clusters, which are self-defined as optimizing some objective function, see \citet{fortunato2009detecting} and \citet{girvan2002community}. In contrast, the blocks under SBM have a model-based definition and do not necessarily have many inner connections (see examples in \citet{daudin2008mixture}). Therefore, most algorithms designed for community detection are generally not suitable in this context.

\citet{bickel2009nonparametric}, \citet{choi2010stochastic}, \citet{celisse2011consistency} and \citet{rohe2010spectral} prove that it is asymptotically possible to uncover the latent structure of the graph $Z$. In this work, we additionally show under a mild assumption that it is possible to do so, just by utilizing degree data instead of the whole graph $X$. As a consequence, we can work with $n$ variables instead of $n^2$, which makes classification computations much faster. The basic reason why so little information is needed ---~compared with other models with latent structure~--- is specific to SBM. The number of observed variables $(X_{ij})_{1\leq i,j\leq n}$ grows faster than the number of latent variables $Z$, therefore even marginal distributions of $X$ concentrate very fast. Our algorithm actually expands the procedure introduced by \citet{snijders1997estimation} when $Q=2$. Like \citet{bickel2009nonparametric}, we provide probabilistic bounds for the occurrence of one error at least. Moreover we take the random assignment into account, even when the number of classes $Q$ increases and the average degree vanishes. Related results are given in \citet{choi2010stochastic} and \citet{rohe2010spectral}. Nevertheless the bounds concern the rate of misclassified nodes instead, and do not prevent the number of errors from growing to infinity as fast as the square root of $n$ for instance. They also require the assignment $Z$ to be fixed.




The paper is organized as follows. In Section \ref{model}, we begin by presenting the model we shall study and some notations are fixed. Above all a concentration property of the degree distribution is stated in paragraph \ref{degreesdist}, which will be very useful in proving the consistency of the method mentioned above. The classification algorithm (called LG) and the main results are presented in this section as well. In particular, Theorem \ref{algoconsistency} provides a bound of the error probability and Proposition \ref{classifrates} gives some convergence rates when the number of classes is allowed to grow. The consistency proof of the LG algorithm is provided in Section \ref{proofs}. Section \ref{plugin} is devoted to deriving simple estimators of the parameters by plug-in and their consistency is also demonstrated. A simulation study in Section \ref{simu} illustrates the behavior of the LG algorithm, which is discussed afterwards. In Section \ref{selection}, the model and the algorithm are more accurately studied. As an application, it is lastly proved that it is likewise possible to find out asymptotically the right number $Q$ of blocks of the model. That completes the method relying just on degrees.

\section{The Stochastic Block Model\label{model}}
\subsection{Model}
We first recall the SBM model. For all integers $n\geq 1$, $[n]$ denotes the set $\{1,\dots,n\}$. The undirected binary graphs with $n$ nodes are defined by the pair $([n],X)$ where $X$ is a symmetric binary square matrix of size $n$. $X$ is called the adjacency matrix of the graph. Let $Q\geq 1$ be the number of blocks.
\begin{itemize}
 \item $Z=(Z_i)_{i\in [n]}$ denotes the \emph{latent} vector of $[Q]^n$ such that $Z_i=q$ if the node $i$ is $q$-labeled. Let $\alpha=(\alpha_1,\dots,\alpha_Q)$ be the vector of the block proportions in the whole population.
\[ Z=(Z_i)_i \text{ i.i.d.} \sim \mathcal{M}(1;\alpha) \]
\item Conditionally on the labels $Z$, the variables $\{X_{ij},i,j\in [n]\}$ are independent Bernoulli variables. Conditionally on $\{Z_i=q,Z_j=r\}$, the parameter of $X_{ij}$ is $\pi_{qr}$.
\[ (X_{ij}|Z_i=q,Z_j=r) \sim \mathcal{B}(\pi_{qr}) \]
\end{itemize}
$\pi_{qr}$ is the connection probability between any $q$-labeled node and any $r$-labeled node. Noting $\pi=(\pi_{qr})_{q,r\in [Q]}$ the connection matrix, the parameters of the model are $(\alpha,\pi)$. This model will be denoted by $\mathcal{G}(n,\pi,\alpha)$. Note that in the sequel $n$ will be often removed in the notations for the sake of simplicity.



This is a classical problem in mixture models: the block labeling is naturally not identifiable. The content of the blocks remains unchanged by permutating labels. But equivalence classes are identifiable as soon as $n\geq 2Q$, see \citet{celisse2011consistency}.

\subsection{Degree distribution} \label{degreesdist}

For all $i\in [n]$, let $D_i^n=\sum_{j\neq i} X_{ij}$ the degree of the node $i$, that is the number of neighbors of this node.

\begin{prop}
For all $q\in [Q]$, let $\overline{\pi}_q=\sum_{r\in[Q]} \alpha_l\pi_{qr}$. $D^n_i$ is a binomial distributed random variable conditionally on $Z_i=q$ with parameters $(n-1,\overline{\pi}_q)$.
\end{prop}
                                                                                                                                                                                                                                                                                                                               
$(D_i^n)_{i\in[n]}$ is therefore a sample of a mixture of binomial distributed random variables with parameters $(n-1,\overline{\pi}_q)_{q\in [Q]} $ and proportions $(\alpha_q)_{q\in[Q]}$.

These variables are correlated. Thus we are not in the validity range of the usual algorithms for mixtures like EM. But there is only one edge shared by any pair of nodes and the degrees are consequently not heavily correlated. Using the EM algorithm would make sense for practical purposes. Nevertheless we have chosen to use a faster one-step algorithm, unlike EM which is iterative.

\subsubsection*{A concentration inequality for binomial random variables}

The following inequality will be useful throughout the article. This will especially account for the fast concentration of the degree distribution. It is a straightforward consequence of Hoeffding's inequality for bounded variables.
 
\begin{theorem}(Hoeffding)
 Let $n\geq 1$, $p\in]0,1[$ and $(Y_i)_{i\in [n]}$ a sequence of independent identically distributed Bernoulli random variables with parameter $p$. Let $S_n=\sum_{i=1}^n Y_i$. Then for all $t>0$:
\begin{equation} \label{cctinq} \tag{CCT}
P\left ( \left | \frac{S_n}{n}-p\right| >t \right) \leq 2e^{-2nt^2} 
\end{equation}
\end{theorem}

\subsubsection*{Concentration property of the normalized degrees}

Define the normalized degree of node $i\in[n]$: \[T_i^n=\frac{D_i^n}{n-1}\]
$(T_i^n)_{i\in[n]}$ cluster around their average conditionally on the node class when $n$ is increasing, according to \eqref{cctinq}:

\begin{equation} \label{cctcond}
 P(|T_i^n-\overline{\pi}_q|>t | Z_i=q) \leq 2e^{-2nt^2}
\end{equation}

Hence normalized degrees corresponding to $q$-labeled nodes gather around $\overline{\pi}_q$. Consequently, in the degree distribution, nodes from different classes split up into groups centered around $\overline{\pi}_q$, provided that all conditional averages $(\overline{\pi}_q)_{q\in [Q]}$ are different. From now on, we will assume that they are:

\paragraph{Assumption}
\begin{equation} \label{hypo} \tag{A}
 \forall q,r\in [Q] \quad q\neq r \Rightarrow \overline{\pi}_q\neq \overline{\pi}_r
\end{equation}

Note that, if it is known that two classes have the same conditional average, it is possible to resort to the concentration of another marginal distribution: the distribution of the number of common neighbors for each pair or nodes. Refer to Appendix \ref{mixed}.


\subsection{Largest Gaps Algorithm}
 Because of the concentration, a larger gap is expected between normalized degrees of nodes from different classes than nodes from the same class. The following algorithm relies on this remark. It consists in building $Q$ blocks by finding the $Q-1$ largest intervals formed by two consecutive normalized degrees.

If $(u_i)_{i\in [n]}$ is a sequence of real numbers, $(u_{(i)})_{i\in [n]}$ denotes the same sequence but sorted in increasing order.

\subsubsection*{Algorithm}
\begin{itemize}
 \item Sort the sequence of the normalized degrees in increasing order: \[ T_{(1)} \leq \dots \leq T_{(n)} \]
  \item Calculate every gap between consecutive normalized degrees: \[ T_{(i+1)} - T_{(i)} \text{ for all } i\in [n-1]\]
  \item Find the indexes of the $Q-1$ largest gaps: $i_1<\dots<i_{Q-1}$, such that for all $k\in [Q-1]$ and for all $i\in[n]\setminus \{i_1,\dots,i_{Q-1}\}$: \[ T_{(i_{k+1})} - T_{(i_{k})} \geq T_{(i+1)} - T_{(i)} \]
  \item Noting $(i_0)=0$ and $(i_Q)=n$, associate with each index $(i)$ a class number: $i\mapsto k$ such that $(i_{k-1})< (i) \leq (i_k)$.
\end{itemize}

\subsubsection*{Example} On the figure below, the largest gaps correspond to the intervals $[T_{(2)},T_{(3)}[$, denoted by \ding{172}, and $[T_{(9)},T_{(10)}[$, denoted by \ding{173}. Nodes $(1)$ and $(2)$ are therefore classified in class 1, nodes from $(3)$ to $(9)$ in 2, nodes $(10)$ and $(11)$ in 3.

\begin{figure}[h]
\begin{center}
\caption{Repartition of the normalized degrees \\ $\Box$: Class 1, $\Diamond$: Class 2, $\bigcirc$: Class 3}
\begin{tikzpicture}
 \draw (0,0) -- (12,0) ;
 \draw (1,0) node {$|$} ;
 \draw (1,-0.5) node {$0$} ;
 \draw (11,0) node {$|$} ;
 \draw (11,-0.5) node {$1$} ;
 \draw (1.86,0) node {$\Box$} ;
 \draw (1.86,0) node {$.$} ;
 \draw (1.86,-0.5) node {$T_{(1)}$} ;
 \draw (2.5,0) node {$\Box$} ;
 \draw (2.5,0) node {$.$} ;
 \draw (2.5,-0.5) node {$T_{(2)}$} ;
 \draw (7,0) node {$\Diamond$} ;
 \draw (7,0) node {$.$} ;
 \draw (6.96,-0.5) node {$T_{(3)}$} ;
 \draw (7.2,0) node {$\Diamond$} ;
 \draw (7.2,0) node {$.$} ;
 \draw (7.3,0) node {$\Diamond$} ;
 \draw (7.3,0) node {$.$} ;
 \draw (7.45,0) node {$\Diamond$} ;
 \draw (7.45,0) node {$.$} ;
 \draw (7.55,-0.5) node {$\dots$} ;
 \draw (7.55,0) node {$\Diamond$} ;
 \draw (7.55,0) node {$.$} ;
 \draw (7.75,0) node {$\Diamond$} ;
 \draw (7.75,0) node {$.$} ;
 \draw (7.97,0) node {$\Diamond$} ;
 \draw (7.97,0) node {$.$} ;
 \draw (8.2,-0.5) node {$T_{(9)}$} ;
 \draw (9.7,-0.01) node {$\bigcirc$} ;
 \draw (9.7,-0.01) node {$.$} ;
 \draw (9.5,-0.5) node {$T_{(10)}$} ;
 \draw (9.9,-0.01) node {$\bigcirc$} ;
 \draw (9.9,-0.01) node {$.$} ;
 \draw (10.3,-0.5) node {$T_{(11)}$} ;
 \draw [<->] (2.5,0.4) -- (7,0.4) ; 
 \draw [<->] (7.97,0.4) -- (9.7,0.4) ;
 \draw (4.75,0.4) node[above] {\ding{172}} ;
 \draw (8.86,0.4) node[above] {\ding{173}} ;
\end{tikzpicture}

\end{center}
\end{figure}

This algorithm has all the qualities mentioned in Introduction and makes good use of the concentration, which makes the consistency easy to prove. Whereas variational EM algorithms runs as many quadratic steps as needed to reach convergence and classical spectral clustering runs in cubic time, this algorithm is especially fast. Indeed the sorting runs in quasilinear time and although the computation of the degrees is quadratic, this is a very basic operation which is very quickly performed. Note that \citet{condon2001algorithms} gave an algorithm running in linear time and consistent under SBM ---~called planted $\ell$-partition model in this paper~---, but provided that the weights of the blocks are equal.

Nevertheless this algorithm seems to be relatively naive because it takes every normalized degree into account and each one carries the same weight, even if it is isolated and not statistically representative. In the worst case, one point is sufficient to trick the algorithm yet makes the classification wrong by a majority, especially at low graph sizes.

\subsection{Main results} \label{mainresult}

The true (respectively estimated) partition of $[n]$ in classes is denoted by the set $\{\mathcal{C}^n_q\}_{q\in[Q]}$, (resp. by $\{\widehat{\mathcal{C}}^n_q\}_{q\in[Q]}$) and the cardinality of the true $q$-labeled class by $N^n_q$ (resp. by $\widehat{N}^n_q$). We expect the estimated partition to be almost surely the true partition when $n$ is large enough. Define $E_n$ as the event ``The LG algorithm makes at least one mistake'', that is:
\[ E_n=\left\{ \{\widehat{\mathcal{C}}_q^n\}_q \neq \{\mathcal{C}_q^n\}_q \right \} \]


\begin{dfn}
 $\{\widehat{\mathcal{C}}^n_q\}_{q\in[Q]}$ is said to be consistent if
 $$P_{\alpha,\pi}^n( E_n ) \xrightarrow[n\to\infty]{} 1$$
\end{dfn}



\begin{dfn}
 Define $\delta$ the characteristic minimal gap (or separability) of the model in the following way: \[\delta=\min\limits_{q\neq r}|\overline{\pi}_q-\overline{\pi}_r|\]
\end{dfn}

Finally, let us define $\alpha_0$ the smallest proportion of the model. The classification is harder for small values of $\alpha_0$: 
\[\alpha_0=\min\limits_{q\in [Q]} \alpha_q \]

\begin{theorem} \label{algoconsistency} Under Assumption \eqref{hypo},
$$P_{\alpha,\pi}^n(E_n) \leq 2ne^{-\frac{1}{8}n\delta^2} + Q (1-\alpha_0)^{n+1}$$
\end{theorem}

Section \ref{proofs} contains the proof of this theorem. The most important parameter is $\delta$: the smaller it is, the harder the separation between the classes is, and so the larger $n$ must be to retrieve the true partition.

\subsubsection*{Convergence rates}

In order to derive orders of magnitude of $n$ to achieve convergence in Theorem \ref{algoconsistency}, we choose another asymptotic framework only in this paragraph, where the parameters are functions of $n$. Consistency does not mean convergence under the distribution of $\mathcal{G}(n,\alpha,\pi)$ anymore, but under $\mathcal{G}(n,\alpha^n,\pi^n)$, with $\alpha^n=(\alpha_1^n,\dots,\alpha^n_{Q_n})$ and $\pi^n=(\pi_{qr}^n)_{1\leq q,r\leq Q_n}$. We assume that:
\[\delta_n\xrightarrow[n\to\infty]{} 0,\: \alpha_0^n \xrightarrow[n\to\infty]{} 0 \text{ and } Q_n\xrightarrow[n\to\infty]{} +\infty \]

\begin{prop} \label{classifrates}
The inference method with LG algorithm is still consistent under the following assumptions:
\begin{enumerate}[(a)]
 \item \label{A1} $\displaystyle \varliminf_{n\to+\infty} \delta_n\sqrt{\frac{n}{\ln n}} > 2\sqrt{2} $
 \item \label{A2} $\displaystyle Q_n=O\left(\sqrt{\frac{n}{\ln n}}\right) $
 \item \label{A3} $\displaystyle \varliminf_{n\to+\infty} -\frac{n\ln(1-\alpha_0^n)}{\ln Q_n} > 1$
\end{enumerate}
For example, if $\displaystyle Q_n=1+\left\lfloor\sqrt{\frac{n}{\ln n}}\right\rfloor$, it is sufficient that: $\displaystyle \alpha_0^n \geq  \frac{\ln n}{2n}$.
\end{prop}

\begin{proof}
Assumption \eqref{A1} implies that there exists $C>2\sqrt{2}$ such that for $n$ large enough:
\[ \delta_n \sqrt{\frac{n}{\ln n}} \geq C \text{ and then } \frac{n\delta_n^2}{\ln n}-8 \geq C^2 -8 >0 \]
Therefore \begin{align*} n\exp\left( -\frac{1}{8}n\delta_n^2 \right) & = \exp \left[ -\frac{1}{8}\ln n\left(\frac{n\delta_n^2}{\ln n}-1 \right) \right] \\
           & \leq \exp\left[ -\frac{1}{8}\ln n\left(C^2 - 8 \right) \right] \xrightarrow[n\to+\infty]{}0
          \end{align*}
Secondly \eqref{hypo} requires $(Q_n-1)\delta_n\leq 1$ as a necessary condition. Hence, applying the first inequality: 
\[ Q_n\leq 1+\frac{1}{\delta_n} =O\left(\sqrt{\frac{n}{\ln n}}\right) \] 

According to Assumption \eqref{A3}, there exists $C'>1$ such that for $n$ large enough:
\[ -\frac{n\ln(1-\alpha_0^n)}{\ln Q_n} > C',\: \text{ so that:} \]

\begin{align*} Q_n(1-\alpha_0^n)^n  & =  \exp\left[ \ln Q_n + n\ln (1-\alpha_0^n) \right ] \\
& = \exp\left[-\ln Q_n \left( \frac{-n\ln(1-\alpha_0^n)}{\ln Q_n} -1\right)\right] \\
& \leq \exp\left(-\ln Q_n \left( C' -1\right)\right) \xrightarrow[n\to+\infty]{} 0
\end{align*}


\end{proof}

Large graphs are more and more sparse as $n$ increases, which results in the decrease in the connectivity defined by $\overline{\overline{\pi}}_n= \esp{\alpha^n,\pi^n}{T_1^n}$.

\begin{prop} The LG algorithm is still consistent in the following cases:
 \begin{itemize}
  \item while $\displaystyle\left(\frac{\ln n}{n}\right)^{3/2}=O(\overline{\overline{\pi}}_n)$, if $Q_n$ is bounded.
  \item while $\displaystyle\sqrt{\frac{\ln n}{n}} = O(\overline{\overline{\pi}}_n)$, if $Q_n\sim\sqrt{\frac{n}{\ln n}}$.
 \end{itemize}

\end{prop}

\begin{proof}
We sketch the proof with the following inequality, which estimates the connectivity of the sparsest model:
\[\displaystyle \overline{\overline{\pi}}_n=\sum_{q=1}^{Q_n} \alpha_q \overline{\pi}_q\geq \sum_{q=1}^{Q_n}
 \alpha_q^n(q-1)\delta_n\geq \alpha_0^n \frac{Q_n(Q_n-1)}{2} \delta_n\]
\end{proof}

\section{Consistency proof of the LG algorithm} \label{proofs}

\subsection{An ideal event for the algorithm}

The LG algorithm delivers the true partition especially when none of the classes is empty, and the spreading of the normalized degrees is small compared with the minimal gap $\delta$. $A_n$ denotes the event ``No true class is empty'', that is \[A_n=\bigcap_{q\in[Q]}\{\mathcal{C}^n_q\neq\varnothing\}= \bigcap_{q\in[Q]}\{N_q^n=0\}\] 

\begin{dfn}
 We call maximal intraclass distance (or spreading) the random variable $d_n$ defined by: \[d_n=\max\limits_{q\in [Q]}\sup\limits_{i\in \mathcal{C}^n_q} |T_i^n-\overline{\pi}_q|\]
\end{dfn}
 This is the maximal distance between the normalized degree of a node and its own conditional mean, over all nodes and all classes. This is basically a measurement of the within-class spreading of the normalized degrees.


\begin{prop} \label{propevent}
Under Assumption \eqref{hypo}, the following inclusion holds for all $\varepsilon>0$: \[ A_n\cap\left\{d_n\leq \frac{\delta}{4+\varepsilon}\right\} \subset \overline{E}_n \]
\end{prop}

\begin{proof}
Suppose that $A_n\cap\{d_n\leq\frac{\delta}{4+\varepsilon}\}$ is true. For all $i,j\in [n]$ and $q,r \in [Q]$:
 \begin{itemize}
\item If nodes $i$ and $j$ have label $q$, then: \[|T_i-T_j|\leq |T_i-\overline{\pi}_q|+|T_j-\overline{\pi}_q|\leq \frac{2\delta}{4+\varepsilon}\]
\item Inversely, if they have different labels, respectively $q$ and $r$, then:
 \begin{align*}
    |T_i-T_j| & \geq |T_j-\overline{\pi}_q|-|T_i-\overline{\pi}_q| \\
	      & \geq |T_j-\overline{\pi}_q| - \frac{\delta}{4+\varepsilon} \\
	      & \geq |       \overline{\pi}_l-       \overline{\pi}_q| -|T_j-       \overline{\pi}_l|-\frac{\delta}{4+\varepsilon} \\
	      & \geq \delta - \frac{\delta}{4+\varepsilon} - \frac{\delta}{4+\varepsilon} = \frac{2+\varepsilon}{4+\varepsilon}\delta > \frac{2\delta}{4+\varepsilon}
   \end{align*}
\end{itemize}
As a conclusion of this alternative, $i$ and $j$ are in the same class if and only if $|T_i-T_j|\leq \frac{2\delta}{4+\varepsilon}$. Notice moreover that there exists exactly $Q-1$ intervals among the set $([T_i,T_j[)_{i,j}$ strictly greater than $\frac{2\delta}{4+\varepsilon}$ on this event. Hence the $Q-1$ largest intervals lie between groups of normalized degrees from different classes; whereas all others lie between degrees of the same class. In this case the algorithm returns the true partition.


\end{proof}

\subsection{Bound of the probability of large spreading}

In this paragraph we shall show that the dispersion $d_n$ converges to 0 thanks to the subgaussian tail of the binomial distributions. This is a basic result of this article, because all others require controlling the dispersion.

\begin{prop} \label{propdispersion}
For all $t>0$:
$$P(d_n>t)\leq 2ne^{-2nt^2}$$
\end{prop}

\begin{proof}
 It consists in conditioning by the class of each node, in order to apply the concentration inequality \eqref{cctinq}, and of a union bound. Since $D_i^n\sim \mathcal{B}(n,\overline{\pi} _q)$, \eqref{cctinq} gave the inequality \eqref{cctcond}:

\[ P(|T_i-\overline{\pi}_q|>t | Z_i=q) \leq 2e^{-2nt^2} \]
Hence:
\begin{align*}
 P(d_n>t) & = \esp{}{P(d_n>t|Z)} \\
       & = \esp{}{P\left( \cup_{q\in [Q]} \cup_{i\in\mathcal{C}_q} \{|T_i-\overline{\pi}_q|>t \} | Z \right)} \\
       & \leq \esp{}{\sum_{q\in [Q]} \sum_{i\in \mathcal{C}_q}P(|T_i-\overline{\pi}_q|>t|Z)} \\
       & \leq \esp{}{\sum_{q\in [Q]} \sum_{i\in \mathcal{C}_q}P(|T_i-\overline{\pi}_q|>t|Z_i=q)} \\
       & \leq 2ne^{-2nt^2} 
\end{align*}

\end{proof}

\paragraph{Remark.} Furthermore $d_n$ almost surely converges to 0 because the upper bound is summable, by applying a usual consequence of the Borel-Cantelli lemma.

\subsection{Bound of the error probability (proof of Theorem \ref{algoconsistency})}

Thanks to the bound of the probability of large spreading, one can easily conclude that the ideal event $A_n\cap\{d\leq \frac{\delta}{4+\varepsilon}\}$ is actually strongly likely for $n$ large enough and for all $\varepsilon>0$:

\begin{proof}
First we have $A_n\cap\{d\leq \frac{\delta}{4+\varepsilon}\} \subset \overline{E}_n $ according to Proposition \ref{propevent}, hence:
$$ P(E_n) \leq P \left ( \overline{A_n\cap\{d_n\leq\frac{\delta}{4+\varepsilon}\}} \right )\leq P\left(d_n>\frac{\delta}{4+\varepsilon}\right) + P(\overline{A_n})$$
On the one hand, Proposition \ref{propdispersion} implies that:
\[ P\left(d_n>\frac{\delta}{4+\varepsilon}\right)\leq 2\exp\left(-2n\left(\frac{\delta}{4+\varepsilon}\right)^2 \right) \]
On the other hand $\overline{A}_n$ corresponds to ``There exists an empty class''. For all $q\in [Q]$, $N_q \sim \mathcal{B}(n,\alpha_q)$, hence:
\begin{align*} P(\overline{A}_n) & =P\left(\cup_{q\in [Q]}\{N_q=0\}\right) \\
& \leq \sum_{q\in [Q]} P(N_q=0) = \sum_{q\in [Q]}(1-\alpha_q)^{n}\leq Q(1-\alpha_0)^{n}.
\end{align*}
Once the both previous inequalities have been put together, we have an upper bound of $P(E_n)$ which depends on $\varepsilon$. The limit of the upper bound when $\varepsilon$ tends to zero yields the bound of the Theorem.
\end{proof}

\section{Consistency of the plug-in estimators\label{plugin}}
If the true classes were known, the usual moment estimators would be enough to estimate $(\alpha,\pi)$. Indeed the empirical proportions estimate $\alpha$ and the connection frequencies estimate the connection probabilities. We first prove that if we knew the classes, we would obtain a consistent estimate. However those variables are not observed but latent. That is why we plug the partition delivered by any consistent classification algorithm into these estimators. Notice that it does not depend on the choice of the consistent algorithm.

\paragraph{Notations} For all $q,r$ in $[Q]$, $\mathcal{C}_{qr}$ denotes $\mathcal{C}_q\times \mathcal{C}_r$, and $N_{qr}$ its cardinality.
If $q\neq r$, $N_{qr}=N_qN_r$ and if $q=r$, $N_{qq}=\frac{N_q(N_q-1)}{2}$. We define the following estimators:
$$\widetilde{\alpha}_q=\frac{N_q}{n} \text{ and }\widetilde{\pi}_{qr}=\frac{1}{N_{qr}}\sum\limits_{(i,j)\in \mathcal{C}_{qr}}X_{ij}$$
Recall that all of these variables are hidden thus far.

\subsection{Estimation with revealed classes}

  \begin{theorem} \label{revealedconsistency}
$(\widetilde{\alpha},\widetilde{\pi})$ is a consistent estimator of $(\alpha,\pi)$.
  \end{theorem}

\begin{proof}
For all $q\in [Q]$, $N_q$ is the sum of $n$ independent Bernoulli random variables with parameter $\alpha_q$. Applying directly the concentration inequality, we get for all $t>0$ and $q\in [Q]$: $P\left( \left| \frac{N_q}{n}-\alpha_q \right | >t \right) \leq 2e^{-2nt^2}$.
Applying the concentration inequality \eqref{cctinq} conditionally on $N_{qr}$ and then taking the expectation, we get for all $t>0$:
$$ P\left( \left| \widetilde{\pi}_{qr}-\pi_{qr} \right |>t \right )=\espc{}{P\left ( \left| \widetilde{\pi}_{qr}-\pi_{qr} \right |>t | N_{qr}\right )} \leq 2\esp{}{e^{-2N_{qr}t^2}} $$
 
 Define:
\[
 \alpha_{qr}=\alpha_q\alpha_r \text{ if } q\neq r \text{ and } \alpha_{qq} = \frac{\alpha_q^2}{2} \text{ if } q=r.
\]
Let $(r_n)$ be a non-negative sequence tending to infinity. We split up the support of the expectation into two pieces, depending on the values of $N_{qr}$. On the one hand the exponential term inside the expectation is bounded on the first piece of the support by a deterministic sequence. On the other hand, the probability of the support of the second piece of the expectation $\left\{|N_{qr}-\alpha_{qr}n^2|>r_n \right\}$ is accurately controlled by using the concentration inequality derived from \eqref{cctinq} in Appendix  \ref{cctprodbinom}.

\begin{align}
\espc{}{\exp(-2N_{qr}t^2)} &= \mathbb{E}\left[  \exp(-2N_{qr}t^2) \mathbbm{1}_{\{  | N_{qr} - \alpha_{qr}n^2 | \leq r_n \} } \right. \notag \\
& \hspace{2cm} \left. +\exp(-2N_{qr}t^2) \mathbbm{1}_{\{  | N_{qr} - \alpha_{qr}n^2 | > r_n \}  }  \right] \notag \\
&\leq \espc{}{\exp(-2t^2(\alpha_{qr}n^2-r_n))} + P(|N_{qr} - \alpha_{qr}n^2 | >r_n ) \notag \\
&\leq \exp(-2t^2(\alpha_{qr}n^2-r_n)) + P\left( \left| \frac{N_{qr}}{n^2} - \alpha_{qr} \right| >\frac{r_n}{n^2} \right) \notag \\
&\leq \exp\left[ - r_n t^2\left( \frac{n^2\alpha_0^2}{r_n}- 1 \right) \right] + 4\exp\left( -\frac12 \frac{r_n^2}{n^3} \right) \label{bound} \tag{B}
\end{align}

In order to have a vanishing bound \eqref{bound}, we just have to choose $(r_n)$ such that:
\[
 \varliminf_{n\to+\infty}\frac{\alpha_0^2 n^2}{r_n} > 1 \text{ and } \frac{r_n^2}{n^3} \xrightarrow[n\to+\infty]{} +\infty
\]

For example, $r_n = n^{7/4}$, hence:
\[
 \espc{}{\exp(-2N_{qr}t^2)} \leq \exp\left[ - n^{7/4} t^2 \left( n^{1/4}\alpha_0^2 - 1 \right) \right] + 4\exp\left( -\frac12\sqrt{n} \right)
\]



Then we conclude with a union bound:

\[
P(\|\widetilde{\pi}-\pi\|_{\infty}>t)\leq 2Q^2\left(e^{ - n^{7/4} t^2 \left( n^{1/4}\alpha_0^2 - 1 \right) }+4e^{-\frac12\sqrt{n}} \right) 
\]

Finally we conclude for all parameters:

\[
P(\|(\widetilde{\pi},\widetilde{\alpha})-(\alpha,\pi)\|_{\infty}>t)\leq 2Q^2\left(e^{ -n^{7/4} t^2 \left( n^{1/4}\alpha_0^2 - 1 \right) }+4e^{-\frac12\sqrt{n}} \right) + 2Qe^{-2nt^2} 
\]

\end{proof}

\subsection{Estimation with hidden classes}

We now assume that we have got a partition of the nodes $\{\widehat{\mathcal{C}}_q\}_q$ returned by any classification algorithm. The estimators $\widehat{\alpha}$ and $\widehat{\pi}$ are defined by plug-in with the estimated partition $\{\widehat{\mathcal{C}}_q\}_q$ instead of the true one $\{\mathcal{C}_q\}_q$. If the classification is right, then estimators both with hat and with tilde are equal.
\[
 \widehat{\alpha}_q=\frac{\widehat{N}_q}{n} \text{ and }\widehat{\pi}_{qr}=\frac{1}{\widehat{N}_{qr}}\sum\limits_{(i,j)\in \widehat{\mathcal{C}}_{qr}}X_{ij}
\]

\begin{theorem} 
If $\{\widehat{\mathcal{C}}_q\}_q$ is consistent, then $(\widehat{\alpha},\widehat{\pi})$ is a consistent estimator of $(\alpha,\pi)$.
\end{theorem}

\begin{proof}
For all $t>0$, let $B_t^n=\{ \left\| (\widehat{\alpha},\widehat{\pi})-(\alpha,\pi)\right\| >t \}$.
\begin{align*}
\forall t>0 \quad P(B_t^n) &= P(B_t^n \cap \overline{E}_n )+ P( B_t^n \cap E_n) \\
& \leq P(B_t^n \cap \overline{E}_n) + P(E_n)
\end{align*}
On the event $\overline{E}_n$, the equality $(\widehat{\alpha},\widehat{\pi})=(\widetilde{\alpha},\widetilde{\pi})$ holds, hence:
\[ \forall t>0 \: P(B_t^n) \leq P\left(\left\|(\widetilde{\alpha},\widetilde{\pi})-(\alpha,\pi)\right\|>t\right) + P(E_n).\]
The first term converges to $0$ according to Theorem \ref{revealedconsistency} and the second one as well, provided the algorithm is consistent (see Theorem \ref{algoconsistency}).
\end{proof}

\subsection{Conclusions}

The previous paragraphs did not depend on the algorithm chosen. Now putting together the results of the previous section and the results concerning the LG algorithm, we get:

\begin{theorem} \label{finalconsistency}
For all $t>0$ 
\begin{multline*} P(\|(\widehat{\pi},\widehat{\alpha})-(\alpha,\pi)\|_{\infty}>t) \leq 2Q^2\left(e^{ -n^2 t^2 \left( \alpha_0^2 - n^{-1/4} \right) }+4e^{-\frac12\sqrt{n}} \right) + 2Qe^{-2nt^2}  \\ + 2ne^{-\frac{1}{8}n\delta^2} + Q (1-\alpha_0)^{n}
\end{multline*}
\end{theorem}

Note that the estimation procedure requires larger graphs to achieve consistency than does the classification procedure with the LG algorithm alone. This is basically due to the variability of the empirical proportions. Since the upper bound is summable, a usual consequence of the Borel-Cantelli lemma implies the strong consistency of these estimators.

\paragraph{Discussion.} We now consider the asymptotic framework $\mathcal{G}(n,\alpha^n,\pi^n)$, as we already did in paragraph \ref{mainresult}. The previous bound above is very interesting when $\varliminf\alpha_0^n>0$ and then $\varlimsup Q_n<+\infty$, because it allows strong consistency for example. If we want just consistency, we can change the bound so that the convergence rates of $(\alpha_0^n)$ and $(Q_n)$ are more optimal in our asymptotic framework.

\begin{prop}
The inference method with LG algorithm is still consistent under Assumptions \eqref{A1}, \eqref{A2} and \eqref{A4}, where

\begin{enumerate}[(a)]
 \setcounter{enumi}{3}
 \item \label{A4} $\displaystyle \varliminf_{n\to+\infty} \alpha_0^n\left(\frac{n}{\ln n}\right)^{1/4} > \sqrt{2}$.
\end{enumerate}

\end{prop}

\begin{proof}
First of all, we consider the bound \eqref{bound} in the proof of Theorem \ref{revealedconsistency}, and this time, we take $r_n = \sqrt{4n^3\ln n}$, so that it yields the following bound:
\begin{multline} \label{bound2} \tag{B'}
 P(\|(\widetilde{\pi},\widetilde{\alpha})-(\alpha^n,\pi^n)\|_{\infty}>t)\leq \\
2Q_n^2\left(\exp\left[ -2t^2\sqrt{n^3\ln n}\left( \left(\alpha_0^n\right)^2 \sqrt{\frac{n}{4\ln n}}-1\right) \right] +\frac{4}{n} \right) + 2Q_ne^{-2nt^2} 
\end{multline}

 Assumption \eqref{A2} is  sufficient to show the convergence of $2Qe^{-2nt^2}$ and $2Q_n^2\frac{4}{n}$. Assumptions \eqref{A2} and \eqref{A4} have to be proved sufficient for the remaining term of the bound \eqref{bound2}. Assumption \eqref{A4} implies that there exists $C>\sqrt{2}$ such that for $n$ large enough:
\[
\alpha_0^n \left({\frac{n}{\ln n}}\right)^{1/4} \geq C, \text{ hence }  (\alpha_0^n)^2 \sqrt{\frac{n}{4\ln n}} -1 \geq \frac{C^2}{2} -1 >0
\]
It is easily deduced from this that the first term of the bound \eqref{bound2} therefore converges to zero. 

Moreover, note that the convergence of this term implies the convergence of $Q_n(1-\alpha_0^n)^n$ as well. Recall that Assumption \eqref{A1} implies the convergence of $2e^{-\frac{1}{8}n\delta_n^2}$. As a conclusion, the consistency holds.

\end{proof}

\section{Simulation study\label{simu}}

Our main purpose in this study is to figure out how the LG algorithm behaves in practice, and above all, to check whether the bounds of Theorem \ref{algoconsistency} are pessimistic or not. The empirical frequency of the graphs with no error would be of great interest, because that is the quantity the bound concerns. But actually this frequency has no smooth evolution: it suddenly shifts from 0 to almost 1. We shall use two types of error rates: a global one and one for each class, so as to examine more accurately the results given by the algorithm.
%
%

\subsection{Simulation design}

The parameters used in the simulation are:

\[ \alpha = (0.3\ 0.6\ 0.1) \quad \pi=\left( \begin{array}{ccc}
                                           0.95 & 0.4 & 0.4  \\
					   0.4 & 0.7 & 0.75   \\
					   0.4 & 0.75 & 0.65
                                          \end{array} \right) \]
Hence $\overline{\pi}=(0.565\ 0.615\ 0.635)$ and $\delta=0.02$.

The evolutions of the classification error rates and the estimators  with respect to the number of nodes $n$ are averaged over 1000 graphs drawn from $\mathcal{G}(n,\alpha,\pi)$ and displayed from 1000 to 60000 nodes.

First of all, the global error rate $g_n$ is defined as the proportion of node pairs $(i,j)$, either classified in distinct classes whereas their true labels are identical, or classified together whereas their true labels are different. That is, denoting $\widehat{Z}$ the label vector returned by the LG algorithm:

\[ g_n(Z,\widehat{Z}) = \frac{2}{n(n-1)}\sum_{1\leq i<j\leq n} \left( \mathbbm{1}_{Z_i=Z_j}\mathbbm{1}_{\widehat{Z_i}\neq\widehat{Z_j}}+\mathbbm{1}_{Z_i\neq Z_j}\mathbbm{1}_{\widehat{Z_i}=\widehat{Z_j}} \right) \]

Secondly, we also propose error rates for each class. Define $I_q$, resp. $M_q$, the rate of intruders (or false positive rate) in the class $q$ predicted by the algorithm, resp. the rate of missing nodes of the true class $q$ (or false negative rate):


\[ I_q^n(Z,\widehat{Z})=\frac{1}{\widehat{N}_q} \sum_{i\in \widehat{\mathcal{C}}_q } \mathbbm{1}_{Z_i\neq q} \text{ and } M_q^n(Z,\widehat{Z})=\frac{1}{N_q} \sum_{i\in \mathcal{C}_q} \mathbbm{1}_{\widehat{Z}_i \neq q} \]

The algorithm gives labels to the nodes in order of increasing degree. Indeed the true labels are expected to be sorted this way, because $\overline{\pi}_1<\overline{\pi}_2<\overline{\pi}_3$. This partially solves the label switching problem which arises when trying to identify the true labels instead of the equivalence classes.

\subsection{Results}

\begin{figure}[h]
  \caption{Evolution of the average global error rate $g_n$ as a function of the graph size}
 \begin{center}
  \includegraphics[scale=0.45]{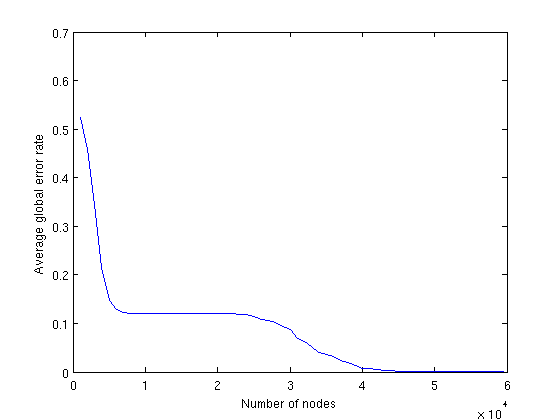}
 \end{center}
\end{figure}

The evolution is quite satisfactory because the error rate completely vanishes from $n=45000$ nodes, which is even earlier than expected from the bound of Theorem \ref{algoconsistency}. Indeed this bound predicted that the probability of at least one error would not be less than 0.05 earlier than $n=300000$. The bound seems to be pessimistic, basically because of the union bound, used in the proof of Proposition \ref{propdispersion}. After a dramatic decrease at the beginning, the evolution encounters a slight stagnation between $n=10000$ and $n=20000$ nodes. An interpretation of this transitional phase can be given with the error rates for each class.


\begin{figure}[h]
\caption{Error rates $I_q^n$ and $M_q^n$}
\hspace{-0.3cm}
\tabcolsep 0.2pt
\begin{tabular}{cc}
 \includegraphics[scale=0.425]{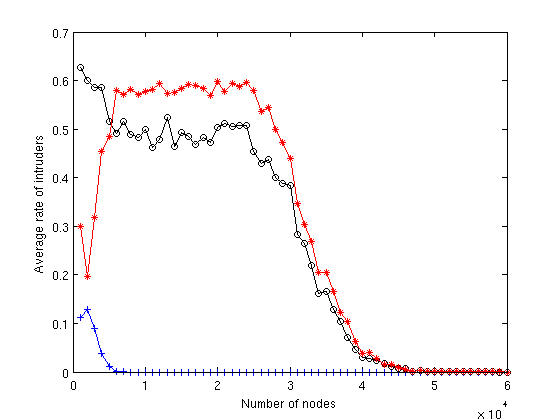} \hspace{-3cm} \raisebox{25ex}{ \begin{tabular}{|c|}
									      \hline
										$\textcolor{blue}{+}$: Class 1 \\
										$\circ$: Class 2 \\
										$\textcolor{red}{*}$: Class 3 \\
									      \hline
									      \end{tabular} } \hspace{0.2cm} &  \includegraphics[scale=0.425]{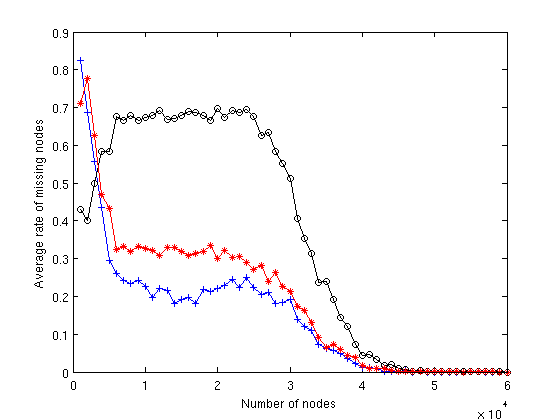}
\end{tabular}
\end{figure}

The first class is much better detected even at low graph sizes, unlike class 2 and class 3. Indeed it is sufficient that the maximal intraclass distance $d_n$ is less than  $(\overline{\pi}_2-\overline{\pi}_1)/4$ to detect this class, whereas the other two are not supposed to be separated before \[ d_n<\frac{\overline{\pi}_3-\overline{\pi}_2}{4} =\frac{\delta}{4}<\frac{\overline{\pi}_2-\overline{\pi}_1}{4} \] according to our previous study. That is the reason why the global error rate dramatically decreases until reaching $n=10000$ nodes, and why it does not vanish before reaching $n=25000$. Note that the bound of Theorems \ref{propdispersion} and \ref{algoconsistency} had not predicted this before reaching $n=50000$ and $n=264000$ respectively.

\begin{figure}[!h]
\caption{Estimators}
\hspace{-1cm}
\tabcolsep 0.2pt
\begin{tabular}{cc}
 \includegraphics[scale=0.425]{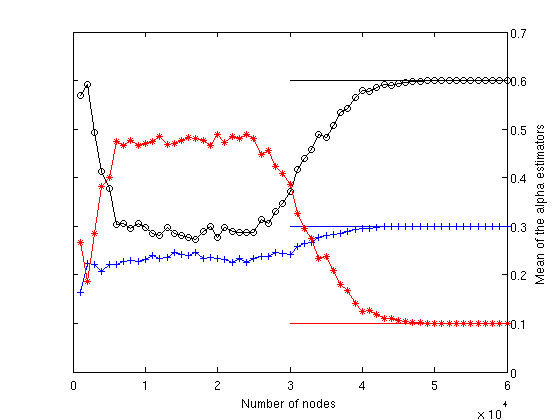} & \includegraphics[scale=0.425]{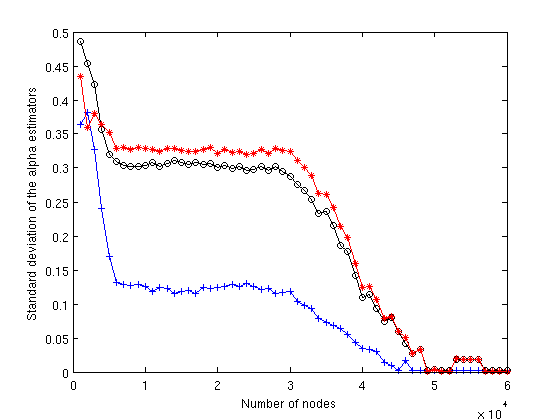}  \hspace{-2.8cm} \raisebox{25ex}
											      { \begin{tabular}{|c|}
											      \hline
												$\textcolor{blue}{+}$: Class 1 \\
												$\circ$: Class 2 \\
												$\textcolor{red}{*}$: Class 3 \\
											      \hline
											      \end{tabular} } \\
Mean of $\widehat{\alpha}$ & Standard deviation of $\widehat{\alpha}$ 
\end{tabular}

\hspace{-1cm}
\begin{tabular}{cc}
 \includegraphics[scale=0.425]{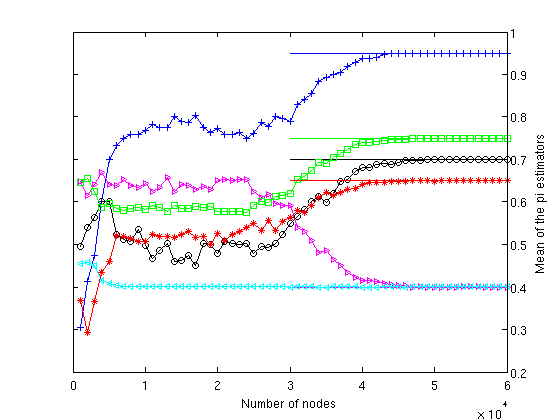} &  \includegraphics[scale=0.425]{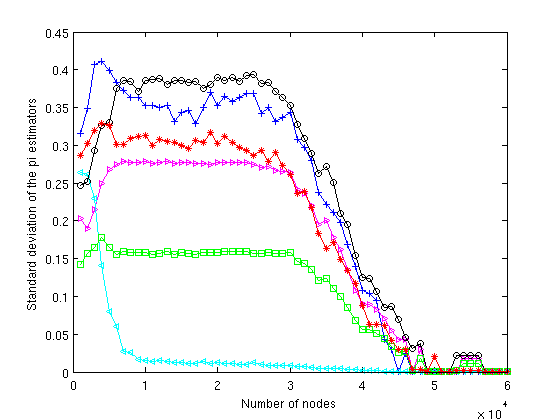} \hspace{-2.2cm}
								\raisebox{21ex}{ \begin{tabular}{|c|}
                                                                \hline
                                                                $\textcolor{blue}{+}$: 1-1 \\
								$\circ$: 2-2 \\
								$\textcolor{red}{*}$: 3-3 \\
								$\textcolor{magenta}{\rhd}$: 1-2 \\
								$\textcolor{cyan}{\lhd}$: 1-3 \\
								$\textcolor{green}{\square}$: 2-3 \\
								\hline
                                                               \end{tabular} } \\
Mean of $\widehat{\pi}$ & Standard deviation of $\widehat{\pi}$
\end{tabular}
\end{figure}

%
%

In short, as long as the tails of the normalized degree distribution are overlapping, the classes are mixed and cannot be properly detected. The curves show in particular that many nodes of class 2 seem to be caught by class 3. Indeed there are many intruders from class 2 in class 3. The missing nodes of class 1 are likely caught by class 2. As a consequence, the proportion of classes 1 and 2 are underestimated in the transitional phase, whereas the proportion of class 3 is overestimated. The inversion of classes 2 and 3 is shown again on graphic 4.1, as on 3.1.

\section{Model selection \label{selection}}

 Up to this section, the number of classes was supposed to be known and was an input parameter of the LG algorithm. Our main purpose hereafter is to examine more accurately the sequence of the gaps sorted in increasing order and then the sequence of the intervals between the means of the groups given by the LG algorithm, depending on the selected number of classes $Q$ for the model. As an application of this study, we finally show that degrees are likewise sufficient to asymptotically select the right number of classes.

\subsection{Study of the gap sequence}


We will use the same notations as in the last section. Moreover $Q_0$ denotes the true number of classes, and $Q$ the current input parameter of the LG algorithm. We will often use the event $B_n=A_n\cap\{d_n\leq\frac{\delta}{5}\}$, where no class is empty and the dispersion $d_n$ is so small that the $Q_0-1$ largest intervals separate the true classes (see Proposition \ref{propevent} with $\varepsilon=1$). Then we can affirm that two normalized degrees are in the same class if and only if their distance is less than $2d_n$. 

Let $(G_q^n)_{q\in [n-1]}$ be the sequence of the distances between consecutive normalized degrees $(T^n_{(i+1)}-T^n_{(i)})_{i\in[n-1]}$, but sorted in decreasing order:
\[
 G^n_1 \geq G^n_2 \geq \dots \geq G^n_{n-1} 
\]
  The $Q_0-1$ largest gaps in the LG algorithm have lengths $G_1,\dots, G_{Q_0-1}$. Define also $(\gamma_q)_{q\in [Q_0-1]}$ the sequence $(\overline{\pi}_{(q+1)}-\overline{\pi}_{(q)})_{q\in [Q_0-1]}$, sorted in decreasing order. This is called the sequence of the theoretical gaps. The following theorem states that largest empirical gaps converge to the corresponding theoretical gaps, which enforces our intuition about the model.

\begin{theorem} \label{theogaps}
 For all $q< Q_0$, $G_q \xrightarrow[n\to+\infty]{} \gamma_q$ a.s.
\end{theorem}

Refer to Appendix \ref{theogapsproof} to see the proof. One can easily realize that the only gap (among the $Q_0-1$ largest) lying between $\overline{\pi}_{(q)}$ and $\overline{\pi}_{(q+1)}$ converges to $\overline{\pi}_{(q+1)}-\overline{\pi}_{(q)}$. However the index of this interval is random and depends on $n$. This interesting but technical problem is solved in the second part of the proof. For the moment we provide a weaker version of this theorem, the proof of which is much simpler. Its conclusion is sufficient for our purposes.


\begin{theorem} \label{theogapsweak}
 For all $q<Q_0$, $\displaystyle \varliminf_{n\to+\infty} G_q > 0$
\end{theorem}

\begin{proof}
 If $q<Q_0$: on the event $B_n$, the $Q_0-1$ largest intervals  necessarily lie between normalized degrees from different classes. There exists $i\in\mathcal{C}_r$ and $j\in\mathcal{C}_s$, where $s\neq r$ such that $G_q=|T_i-T_j|$. But $|T_i - \overline{\pi}_r|\leq d_n$ and $|T_j - \overline{\pi}_s|\leq d_n$, hence
\[G_q\geq |\overline{\pi}_r - \overline{\pi}_s| - 2d_n \geq \delta - \frac25 \delta = \frac35\delta >0\]
Namely $B_n \subset \{G_q \geq \frac35\delta\}$. 
\[ P\left( G_q< \frac35 \delta \right) \leq P\left(\overline{B}_n\right)\leq 2e^{-\frac{2}{25}n\delta^2} + Q_0(1-\alpha_0)^{n} \]
As the upper bound is summable, according to the Borel-Cantelli lemma, \[P\left( \varlimsup_{n\to+\infty}\{G_q<\frac35 \delta\} \right) = 0\] Therefore $\varliminf_{n \to+\infty} G_q \geq \frac35\delta>0$ almost surely.

\end{proof}

All further gaps lie between degrees of nodes of the same class and then converge to zero. The next theorem gives an estimation of the convergence rate.

\begin{theorem} \label{theogaps2} 
 For all $\beta\in]0,1[$, the triangular array \[\{n^{\frac{1-\beta}{2}}G_q^n; Q_0\leq q\leq n-1\}\] converges uniformly w.r.t. $q$ and a.s. to zero when $n$ tends to infinity.
\end{theorem}

\begin{proof}

First of all, recall that for all $n$,
\[ G^n_{Q_0} \geq G^n_{Q_0+1} \geq \dots \geq G^n_{n-1}\geq 0\]
Therefore we can just prove that $n^{\frac{1-\beta}{2}}G_{Q_0}\xrightarrow[n\to+\infty]{} 0$, and the uniform convergence will follow.

On the event $B_n$, the $Q_0-1$ largest intervals lie between normalized degrees from different classes. The next intervals lie between degrees from the same class, and the distance to their corresponding conditional mean is at most $d_n$. As $G_{Q_0}$ is one of these, $G_{Q_0}\leq 2d_n$. Hence, for all $0<t<\frac{\delta}{5}$:
\begin{align*}
 P\left(n^{\frac{1-\beta}{2}}G_{Q_0}>t \right) &= P(n^{\frac{1-\beta}{2}}G_{Q_0}>t \cap B_n) + P\left( n^{\frac{1-\beta}{2}}G_{Q_0}>t \cap \overline{B}_n \right) \\
&\leq P(2n^{\frac{1-\beta}{2}}d_n>t) + P\left(\overline{B}_n \right) \\
&\leq 2(e^{-\frac12 n^{\beta}t^2} + e^{-\frac{2}{25}n\delta^2}) + Q_0(1-\alpha_0)^{n}
\end{align*}

\end{proof}

\subsection{Study of the intervals between estimated classes}

By distances between estimated classes, we mean distances between empirical averages of the normalized degrees of each class, provided by the LG algorithm. Define $m_q$ to be the average of the normalized degrees of the $q$-labeled class estimated by the algorithm: \[m_q=\frac{1}{N_q}\sum_{i\in \widehat{\mathcal{C}}_q} T_i \] The sequence of the gaps between consecutive averages $(m_{(q+1)}-m_{(q)})_{q\in [Q-1]}$ is sorted in order of decreasing length, just as the sequence of the gaps $(T_{(i+1)}-T_{(i)})_{i\in [n-1]}$ is in the previous paragraph. This new sequence is denoted by $(H^n_q)_{q\in [Q-1]}$. Of course it depends on the current $Q$, whereas $(G_q)_q$ does not.

When $Q=Q_0$, $H_q$ and $G_q$ are very close for all $q\leq Q_0-1$. On the contrary, when $Q<Q_0$, some of the $(H_q)_{q\in [Q_0-1]}$ stretch over several classes and include more than one of the $G_q$. As a result, there is at least one $q$ such that $H_q$ asymptotically differs from $G_q$.


\begin{theorem} \label{theogapsmeans}
\
 \begin{enumerate}
  \item If $Q=Q_0$, then $\sum\limits_{q=1}^{Q-1} ( H_q - G_q ) \xrightarrow[n\to+\infty]{a.s.} 0$
  \item If $Q<Q_0$, then $\varliminf\limits_{n\to+\infty} \sum\limits_{q=1}^{Q-1} (H_q - G_q)  > 0 $ a.s.
 \end{enumerate}
\end{theorem}

\begin{proof}
Let $(J_q)_{q\in [Q_0-1]}$ the $Q_0-1$ largest intervals between consecutive normalized degrees, hence for all $q$, $|J_q|=G_q$. Define also $J_0'=[0,\min_{i\in [n]} T_i[$ and $J_Q'=[\max_{ i\in [n]} T_i,1[$. The union of $J_0',J_1,\dots,J_{Q-1},J_Q'$ partially covers the interval $[0,1[$. These intervals are separated and the distance between the bounds of consecutive intervals is at most $2d_n$. As a result:
\[ 1- 2Q_0d_n \leq \sum_{q=1}^{Q_0-1} G_q +H_0 + H_Q \leq 1 = \sum_{q=0}^Q H_q \]

\begin{description}
 \item[$Q=Q_0$] Subtracting the right-hand side (which actually equals 1), we deduce from both previous inequalities that:
\[ -2Q_0d_n \leq \sum_{q=1}^{Q_0-1} (G_q - H_q) \leq 0 \]
The first assertion follows directly from this inequality; for all $t>0$:
\begin{align*}
 P\left( \left|\sum_{q=1}^{Q_0-1} (H_q - G_q) \right| > t \right)  &\leq P( 2Q_0d_n >t ) \\
& \leq 2\exp\left( -2n\left(\frac{t}{2Q_0}\right)^2 \right) = 2 \exp\left( -\frac1{2Q_0^2} n t^2 \right)
\end{align*}
  \item[$Q<Q_0$] Subtracting the right-hand side from the second inequality only yields this time:
\[ \sum_{q=Q}^{Q_0-1} G_q \leq \sum_{q=1}^{Q-1} (H_q - G_q) \]
But as shown in Theorem \ref{theogapsweak}, the lower limit of $G_q$ is non-negative for all $q\leq Q_0-1$. \emph{A fortiori}, the second assertion of the theorem \ref{theogapsmeans} stands as well.
\end{description}

\end{proof}

\subsection{Application to model selection}

The summed differences $\sum_{q=1}^{Q-1} (H_q - G_q)$  examined in the last paragraph have an interesting property regarding model selection: when $Q$ is the right number of classes, it converges to zero, and when $Q$ is too small, it converges to a non-negative value, because one of the $H_q$ does not match $G_q$. Thus this quantity measures the risk of underestimating the number of classes.

However, its minimization over all $Q\in\{2,\dots,n\}$ yields the unexpected solution $Q=n$, for all $Q_0$. Therefore we have to penalize overly small gaps between normalized degrees. We chose to use an \emph{ad hoc} penalty, that can be easily inferred from our previous study, in order to have a correct estimate of $Q_0$. Define for all $Q\in \{2,\dots,n\}$:

\[ f_Q = \sum_{q=1}^{Q-1} (H_q - G_q) + \frac{1}{n^{\frac{1-\beta}{2}}G_{Q-1}} \in [0,+\infty] \text{ where } \beta\in]0,1[. \]

\begin{theorem}
\
 \begin{enumerate}
  \item If $Q=Q_0$, then   $f_Q \xrightarrow[n\to+\infty]{a.s.} 0 $
  \item If $Q<Q_0$, then   $\varliminf\limits_{n\to+\infty} f_Q >0 $ a.s.
  \item If $Q>Q_0$, then $f_Q \xrightarrow[n\to+\infty]{a.s.} +\infty $
 \end{enumerate}
It follows that $\widehat{Q}=\arg\min\limits_{ 2 \leq Q \leq n } f_Q \xrightarrow[n\to+\infty]{} Q_0$ a.s.
\end{theorem}

\begin{proof}

 \begin{description} 
  \item[If $Q=Q_0$] Applying Theorem \ref{theogapsmeans}, the sum $\sum_{q=1}^{Q-1} (H_q - G_q)$ converges a.s. to 0. According to Theorem \ref{theogapsweak}, $\varliminf\limits_{n\to+\infty} G_{Q_0-1} >0$ almost surely. Therefore:
\[ \frac{1}{n^{\frac{1-\beta}{2}}G_{Q-1}} \xrightarrow[n\to+\infty]{a.s.} 0,
 \text{ and then } f_Q \xrightarrow[n\to+\infty]{a.s.} 0
\]
  \item[If $Q<Q_0$] According to the second assertion of Theorem \ref{theogapsmeans}, the lower limit of the first term is non-negative. There is no change by adding the second term, because it is positive. Hence: \[\varliminf_{n\to+\infty} f_Q >0 \]
  \item[If $Q>Q_0$] The sum $\sum_{q=1}^{Q-1} H_q - G_q$ is lower bounded by -1 (notice that it is even positive), and according to the second assertion of Theorem \ref{theogaps2}, $(n^{\frac{1-\beta}{2}}G_{Q-1})_n$ uniformly converges to 0, as soon as $q\geq Q_0$. The last assertion follows.
 \end{description}

\end{proof}

\section{Conclusions}

Unlike most of the methods known thus far, the LG algorithm is able to process very large graphs. In fact it provides good results only for such graphs. However, in practice, the algorithm is efficient even for smaller graphs than theoretically expected. Moreover it is self-sufficient: it provides consistent methods for node clustering, parameter estimation and model selection. Lastly, this algorithm is free from any preliminary setting. Consequently there is need neither for any prior knowledge nor for multiple runnings of the algorithm. Thus it can quickly provide good initialization values for other algorithms which depend severely on them.

Above all, the LG algorithm performs every task using the degree data alone. As a conclusion, the degree data asymptotically includes the information needed to achieve all of the statistical inference in this model.

\appendix

\section{Concentration inequality for products of binomial distributed variables} \label{cctprodbinom}

\begin{prop}
 Let $X$ (respectively $Y$) be a sum of $n$ independent bernoulli distributed variables with parameter $p$, respectively $q$. Then for all $t>0$
$$ P\left(\left|\frac{XY}{n^2}-pq\right|>t\right) \leq 4\exp(-\frac{1}{2}nt^2)$$
\end{prop}

\begin{proof}
 \begin{align*}
  P\left(\left|\frac{XY}{n^2}-pq\right|>t\right) &= P\left(\left|\left(\frac{X}{n}-p\right)\frac{Y}{n}+\left(\frac{Y}{n}-q\right)p\right|>t\right) \\
&\leq P\left(\left|\frac{X}{n}-p\right|\frac{Y}{n}>\frac{t}{2}\right)+ P\left(\left|\frac{Y}{n}-q\right|p >\frac{t}{2}\right) \\
&\leq P\left(\left|\frac{X}{n}-p\right|>\frac{t}{2}\right)+ P\left(\left|\frac{Y}{n}-q\right| >\frac{t}{2}\right) \\
&\leq 2\times2\exp\left(-2n\left(\frac{t}{2}\right)^2\right) = 4\exp(-\frac12 nt^2)
 \end{align*}
The last line is obtained by applying the usual concentration inequality \eqref{cctinq} to both $X$ and $Y$.
\end{proof}

With a similar proof, we prove that for all $t\in]0,1/4]$:

\[
 P\left( \left| \frac{X(X-1)}{2n^2} -\frac{\alpha^2}{2} \right| >t \right) \leq 4\exp\left( {-2nt^2} \right)
\]

\section{Separation of mixed classes\label{mixed}}


Suppose that there are $Q$ classes and $\overline{\pi}_q=\overline{\pi}_r$ for some $q$ and $r$. For the sake of simplicity, all other conditional averages are assumed to be pairwise distinct. The LG algorithm is supposed to be previously applied to the graph with the input parameter $Q-1$. Let us point out that the $Q-1$ groups returned by LG are asymptotically the true classes, except classes $q$ and $r$, which are mixed together in one group of nodes, denoted by $M$. We shall briefly explain a procedure to separate this group, using the concentration of some additional binomial variables, namely the number of common neighbors of each pair of nodes.

\paragraph{Notation.}
Define $\underline{\alpha}$ the diagonal matrix the diagonal coefficients of which are $(\alpha_q)_{q\in [Q]}$ and the bilinear map on $\mathbb{R}^Q$:
\[ \langle\cdot,\cdot\rangle_{\alpha}:\: (X,Y) \mapsto {}^t\! X \underline{\alpha}Y \]
which is a scalar product, as soon as $\alpha_q$ is non-negative for all $q$. $\|\cdot\|_{\alpha}$ denotes the associated norm.

\paragraph{}

For all pairs of nodes $(i,j)\in M\times M$, define \[ D_{ij} = \sum_{k\neq i,j} Y_{ijk}, \text{ where } Y_{ijk} = X_{ik}X_{jk}. \]
$Y_{ijk}$ is a Bernoulli distributed variable, that equals one if and only if $i$ \emph{and} $j$ are both connected to $k$. Its parameter conditionally depends on each class of nodes $i$ and $j$:
\begin{itemize}
 \item If $i$ and $j$ both belong to the $q$-labeled class: \[ P(Y_{ijk}=1|Z_i=Z_j=q) = \sum_{l=1} \alpha_l \pi_{ql}^2  = \|\pi_q\|_{\alpha}^2 \] where $\pi_q$ is the row vector $(\pi_{ql})_l$. Symmetrically, if they both belong to the $r$-labeled class, the parameter is $\|\pi_r\|_{\alpha}^2$.
  \item Otherwise, if they belong to distinct classes $q\neq r$: \[ P(Y_{ijk}=1|Z_i=q,Z_j=r) =\sum_{l=1}^Q \alpha_l\pi_{ql}\pi_{rl} = \langle\pi_q,\pi_r\rangle_{\alpha}\]
\end{itemize}

The behavior of the new variables $D_{ij}$ looks like that of the degrees; they once more quickly concentrate around their average value as a consequence of the concentration of binomial variables. There are three groups of node pairs, concentrating around $\|\pi_q\|_{\alpha}^2$, $\|\pi_r\|_{\alpha}^2$, or $\langle\pi_q,\pi_r\rangle_{\alpha}$. 

Suppose that $\|\pi_q\|_{\alpha}\leq \|\pi_r\|_{\alpha}$. Applying the Cauchy-Schwarz inequality, \[ 0\leq \langle\pi_q,\pi_r\rangle_{\alpha} \leq \|\pi_q\|_{\alpha}\|\pi_r\|_{\alpha} \leq \|\pi_r\|_{\alpha}^2 \]

The case of equality in the Cauchy-Schwarz inequality cannot arise; if it did, then $\pi_q$ and $\pi_r$ would be collinear vectors. Noting $c$ the constant of collinearity, we would get $\overline{\pi}_q = c\overline{\pi}_r$. But $\overline{\pi}_q$ and $\overline{\pi}_r$ are assumed to be equal in this section; hence $c=1$. $\pi_q$ and $\pi_r$ would be equal. This is not allowed by the model for identifiability reasons. The inequality is finally strict, which especially implies:
\[ 0\leq \langle\pi_q,\pi_r\rangle_{\alpha} < \|\pi_r\|_{\alpha}^2 \]

The furthest group to the right on the real line consequently contains only pairs of nodes of the same membership, which is sufficient to solve the mixing problem. We just have to extract this group from the other two by using the LG algorithm with $Q=2$ as input parameter. Define $W$ as the set of the pairs which are in this group, and $F$ as the set of nodes, which are involved in those pairs. Let $K$ be the graph defined by $(F,W)$. There are three cases:

\begin{itemize}
 \item If $\langle\pi_q,\pi_r\rangle_{\alpha} \leq \|\pi_q\|_{\alpha}\leq \|\pi_r\|_{\alpha}$ and $\|\pi_q\|_{\alpha}-\langle\pi_q,\pi_r\rangle_{\alpha} < \|\pi_r\|_{\alpha} - \|\pi_q\|_{\alpha}$, $K$ asymptotically forms one clique composed of all nodes from the $r$-labeled class. Hence we deduce that remaining nodes are from the $q$-labeled class.
 \item If $\langle\pi_q,\pi_r\rangle_{\alpha} < \|\pi_q\|_{\alpha} \leq \|\pi_r\|_{\alpha}$ and $\|\pi_q\|_{\alpha}-\langle\pi_q,\pi_r\rangle_{\alpha} > \|\pi_r\|_{\alpha} - \|\pi_q\|_{\alpha}$, then the graph $K$ has asymptotically two cliques: one formed by the nodes of class $q$ and the other one by the nodes of class $r$. If the equality holds in the second inequality, there is either one clique as in the first case or two, depending on the selected gap.
 \item If $\|\pi_q\|_{\alpha} < \langle\pi_q,\pi_r\rangle_{\alpha} < \|\pi_r\|_{\alpha} $, the gap between $\|\pi_q\|_{\alpha}$ and $\langle\pi_q,\pi_r\rangle_{\alpha}$ is necessarily strictly shorter than the one between $\langle\pi_q,\pi_r\rangle_{\alpha}$ and $\|\pi_r\|_{\alpha}$. Indeed this amounts to saying that $\|\pi_q-\pi_r\|^2>0$. Thus $K$ asymptotically forms one clique again. 
\end{itemize}

\section{Proof of Theorem \ref{theogaps}} \label{theogapsproof}

Let us define $(J_i)_{i\in [n]}$ the sequence of the intervals $[T_{(i)},T_{(i+1)}[$ sorted in order of decreasing length, hence for all $i\in [n]$, $|J_i|=G_i$. We suppose hereafter that the sequence $(\overline{\pi}_q)_q$ is sorted in increasing order: $\overline{\pi}_1 < \dots < \overline{\pi}_Q$.


\begin{proof}

On the event $B_n$, among the $Q_0-1$ largest intervals, we can associate with each $\overline{\pi}_q$ the only one lying between $\overline{\pi}_q$ and $\overline{\pi}_{q+1}$. Namely the only $J_i$ with $i\in[Q_0-1]$ such that $J_i \cap ]\overline{\pi}_q,\overline{\pi}_{q+1}[ \neq \varnothing $. $S(q)$ denotes the index in $[Q_0-1]$ corresponding to this unique interval.

Moreover, $s(q)$ denotes one of the indexes $s\in[Q_0-1]$ such that $\gamma_s= \overline{\pi}_{q+1}-\overline{\pi}_{q}$, chosen so that $s$ is injective. Let us point out that $S$ is a random permutation whereas $s$ is deterministic. In order to simplify notations, we silently make the deterministic index change $r=s(q)$. Thereby $(\gamma_q)_q$ still denotes the sequence $(\gamma_{s(q)})_q$, and $S$ the permutation $S\circ s^{-1}$.

Notice that on $B_n$ and especially when $d_n\leq \frac{\delta}{5}$:
\[ [\overline{\pi}_q+d_n, \overline{\pi}_{q+1}-d_n] \subset J_{S(q)} \subset [\overline{\pi}_q-d_n, \overline{\pi}_{q+1}+d_n] \]
\begin{equation} \text{Hence } |G_{S(q)}-\gamma_q|\leq 2d_n. \label{ineg} \end{equation}
\paragraph{1.} We first prove that the gap $G_{S(q)}$ converges to the theoretical gap $\gamma_q$. For all $t>0$:
\begin{align}
P(|G_{S(q)} - \gamma_q| >t) &= P(|G_{S(q)} - \gamma_q|>t \cap B_n ) + P(|G_{S(q)} - \gamma_q| >t \cap \overline{B}_n) \notag \\
&\leq P( 2d_n>t ) + P(\overline{B}_n) \notag \\
&\leq 2(e^{-\frac12 nt^2} + e^{ -\frac{2}{25} n \delta^2} ) +Q_0(1-\alpha_0)^{n} \label{ineg2}
\end{align}
\paragraph{2.} Secondly, none of the $Q_0-1$ largest intervals permute anymore expect for those having the same theoretical values. It follows from the inequality \eqref{ineg} that for all $q,r\in[Q_0-1]$,
\[ \gamma_q - \gamma_r -4d_n \leq G_{S(q)} - G_{S(r)} \leq \gamma_q - \gamma_r +4d_n \]
Define $\eta = \frac15 (\min_{q\in [Q]} (\gamma_q - \gamma_{q+1})\wedge\delta)$, a threshold designed to distinguish distances converging to one value from those converging to another. On the event $d_n\leq \eta$, the previous inequality yields:
\[ \gamma_q - \gamma_r - 4\eta \leq G_{S(q)} - G_{S(r)} \leq \gamma_q - \gamma_r +4\eta \]
\begin{itemize}
 \item If $\gamma_q-\gamma_r<0$, then $\gamma_q-\gamma_r+4\eta < 0$ is also true by the definition of $\eta$. As a result of the inequality just above, $G_{S(q)}-G_{S(r)}<0$.
\item If $\gamma_q-\gamma_r>0$, then $\gamma_q-\gamma_r-4\eta > 0$, and $G_{S(q)}-G_{S(r)}>0$.
\end{itemize}

If $(u_i)_{1\leq i\leq m}$ is a sequence, we write $i\sim_u j$ if and only if $u_i=u_j$. $\sim_u$ is an equivalence relation. Applying the Lemma \ref{lemme} stated and proved afterwards, if $d_n\leq \eta$, there exists $r\sim_{\gamma}q$ such that $q=S(r)$. Notice furthermore that the sequence $(\gamma_q)_{q\in [Q_0-1]}$ is constant on the $\sim_{\gamma}$-equivalence classes. The term $|G_q - \gamma_q|$ is necessarily in the sum $\sum_{r\sim q} |G_{S(r)} - \gamma_r|$. Finally, define 
\begin{align*}
 P(|G_q - \gamma_q| > t ) &= P(|G_q - \gamma_q| > t \cap B_n ) + P(|G_q - \gamma_q| > t \cap \overline{B}_n) \\
&\leq P\left(\sum_{r\sim q} |G_{S(r)} - \gamma_r|>t \right) + P(\overline{B}_n)  \\
&\leq \sum_{r\sim q} P\left(|G_{S(r)} - \gamma_r|> \frac{t}{Q_0}\right) + P(\overline{B}_n) \\ 
&\leq 2Q_0(e^{-\frac1{2Q_0^2} nt^2} + e^{ -\frac{2}{25} n \delta^2} ) + 2e^{-2n\eta^2} \text{ according to }\eqref{ineg2}.
\end{align*}
\end{proof}

\begin{lemma} \label{lemme}
 Let $(u_i)_{1\leq i\leq m},(v_i)_{1\leq i\leq m}$ be two real decreasing sequences. Let $p$ be the number of $\sim_u$-equivalence classes and $\sigma$ one permutation of $\{1,\dots,m\}$. We especially assume that for all $i,j\in\{1,\dots,m\}$,
\begin{itemize}
 \item $u_i < u_j \Rightarrow v_{\sigma(i)} < v_{\sigma(j)}$
 \item $u_i > u_j \Rightarrow v_{\sigma(i)} > v_{\sigma(j)}$
\end{itemize}
Then $\sigma = \sigma_1 \circ \dots \circ \sigma_p$ where the support of $\sigma_i$ is the $i^{th}$ $\sim_u$-equivalence class.

\end{lemma}

\begin{proof}
Since $u$ is decreasing, the $\sim_u$-equivalence classes are just sets of consecutive natural integers. Define recursively $(r_i)_{1\leq i\leq p}$ the increasing sequence of indexes $j$ when the value of $u_j$ changes:
\begin{itemize}
 \item Let $r_1=1$.
  \item For $i\geq 1$, let $r_{i+1}$ be the smallest integer $j > r_i$ such that $u_{r_i}=\dots=u_{j-1} > u_j$.
\end{itemize}
The construction of $(r_i)_i$ implies that for all $j < r_i$, all $ r_i \leq l< r_{i+1}$ and all $k \geq r_{i+1}$: $ u_j < u_k <u_l$, and furthermore $v_{\sigma(j)}< v_{\sigma(k)} < v_{\sigma(l)}$ as well. As $v$ decreases, $\sigma(\{r_i,\dots,r_{i+1}-1\})=\{r_i,\dots,r_{i+1}-1\}$. The result follows directly from this.

\end{proof}

\newpage

\bibliographystyle{plainnat}
\bibliography{bibliographie}

\end{document}